\documentclass[11pt]{article}

\usepackage{amssymb,amsmath,amstext,amscd,amsthm,latexsym}

\usepackage[dvipdfmx, bookmarkstype=toc, colorlinks=true, linkcolor=cyan, pdfborder={0 0 0}, bookmarks=true, bookmarksnumbered=true]{hyperref}

\DeclareMathOperator*{\regprod}{\mathchoice%
{\ooalign{\hbox{$\displaystyle\prod$}\crcr\hbox{$\displaystyle\coprod$}}}
{\ooalign{\hbox{$\textstyle\prod$}\crcr\hbox{$\textstyle\coprod$}}}
{\ooalign{\hbox{$\scriptstyle\prod$}\crcr\hbox{$\scriptstyle\coprod$}}}
{\ooalign{\hbox{$\scriptscriptstyle\prod$}\crcr\hbox{$\scriptscriptstyle\coprod$}}}
}

\numberwithin{equation}{section}

\theoremstyle{plain}

\newtheorem{theorem}{Theorem}

\newtheorem{lemma}{Lemma}

\newtheorem{example}{Example}

\newtheorem{definition}{Definition}

\newtheorem{corollary}{Corollary}

\begin{document}

\title{Riemann operators on higher $K$-groups}

\author{Nobushige Kurokawa\footnote{Department of Mathematics, Tokyo Institute of Technology} \and Hidekazu Tanaka\footnote{6-15-11-202 Otsuka, Bunkyo-ku, Tokyo}}

\date{September 19, 2022}



\maketitle

\begin{abstract}
We introduce Riemann operators acting on Quillen's higher $K$--groups $K_{n}(A)$ for the integer ring $A$ of an algebraic number field $K$. Especially we prove that gamma factors of Dedekind zeta function of $K$ are obtained as regularized determinants of Riemann operators.
\end{abstract}

\section*{Introduction} In this paper we introduce ``Riemann operators'' $\mathcal{R}$ acting on Quillen's higher $K$--groups $K_{n}(A)$ for the integer ring $A$ of an algebraic number field $K$. We expect that such Riemann operators would supply determinant expressions for zeta allied functions. In fact, here we show that so called gamma factors of the Dedekind zeta function $\zeta_{A}(s)$ --- the Hasse zeta function of Spec($A$) --- are naturally obtained as the zeta regularized determinants of $\mathcal{R}$.
\par We remark that certain different ``Riemann operators'' so called ``absolute Frobenius operators'' were studied by Deninger \cite{D1} \cite{D2} around 30 years ago: see Manin \cite{M} for an excellent survey lectured in 1992. 
\par Now we define $\mathcal{R}$ by

\begin{definition} Let $\mathcal{R} | K_{n}(A)_{\mathbb{C}}=\frac{1-n}{2}$, where $K_{n}(A)_{\mathbb{C}}=K_{n}(A) \otimes_{\mathbb{Z}} \mathbb{C}$.
\end{definition}

Then we prove

\begin{theorem} 
\[
{\rm det}((s I-\mathcal{R}) | \bigoplus_{n=0}^{\infty} K_{n}(A)_{\mathbb{C}} ) = (s-\frac{1}{2})s^{-1}\biggl[ \Gamma_{\mathbb{R}}(s)^{r_{1}} \Gamma_{\mathbb{C}}(s)^{r_{2}} (\sqrt{2\pi})^{[K:\mathbb{Q}]s} \biggl]^{-1} C(K),
\]
where $C(K)$ is a constant given as
\[
C(K)=(2\sqrt{\pi})^{r_{1}} (2\sqrt{2\pi})^{r_{2}}.
\]
\end{theorem}

\par We notice that the determinant means the zeta regularized one:
\begin{align*}
{\rm det}((s I-\mathcal{R}) | \bigoplus_{n=0}^{\infty} K_{n}(A)_{\mathbb{C}} ) &= \regprod_{n=0}^{\infty} (s-\frac{1-n}{2})^{{\rm dim} K_{n}(A)_{\mathbb{C}}}\\
&=\exp( -\frac{\partial}{\partial w} \varphi_{A}(w,s)\biggl|_{w=0} ),
\end{align*}
where $\varphi_{A}(w,s)=\sum_{n=0}^{\infty}{\rm dim}K_{n}(A)_{\mathbb{C}}(s+\frac{n-1}{2})^{-w}.$ Moreover, as usual, $r_{1}$ and $r_{2}$ are determined by the isomorphism of $\mathbb{R}$--algebras
\[
K \otimes_{\mathbb{Q}} \mathbb{R} \cong \mathbb{R}^{r_{1}} \times \mathbb{C}^{r_{2}}. 
\] 
We recall that the functional equation of the Dedekind zeta function $\zeta_{A}(s)$ is given as
\[
\widehat{\zeta}_{A}(x)=\widehat{\zeta}_{A}(1-s),
\]
where
\[
\widehat{\zeta}_{A}(x)=\zeta_{A}(s)\Gamma_{\mathbb{R}}(s)^{r_{1}}\Gamma_{\mathbb{C}}(s)^{r_{2}}|D(K)|^{\frac{s}{2}}
\]
with
\begin{align*}
\Gamma_{\mathbb{R}}(s) &= \Gamma(\frac{s}{2})\pi^{-\frac{s}{2}},\\
\Gamma_{\mathbb{C}}(s) &= \Gamma(s)2(2\pi)^{-s}
\end{align*}
and $D(K)$ being the Discriminant of $K$ overe $\mathbb{Q}$. By letting $K=\mathbb{Q}$ in Theorem 1 we get the following

\begin{corollary}
\[
{\rm det}((s I-\mathcal{R}) | \bigoplus_{n=0}^{\infty} K_{n}(\mathbb{Z})_{\mathbb{C}} ) = 2 \sqrt{\pi} (s-\frac{1}{2}) s^{-1} \Gamma_{\mathbb{R}}(s)^{-1},
\]
where $\Gamma_{\mathbb{R}}(s)$ is the gamma factor of the Riemann zeta function $\zeta_{\mathbb{Z}}(s)$.
\end{corollary}

\section{Zeta regularized products} We recall zeta regularized products
\[
\regprod_{\lambda}(\lambda)^{a(\lambda)} = \exp (-\varphi'(0)),
\]
where
\[
\varphi(w) = \sum_{\lambda} a(\lambda) \lambda^{-w}.
\]
We refer \cite{D1} \cite{D2} \cite{M} \cite{KT} for details.

\begin{example}
\[
\regprod_{n=1}^{\infty} n = \sqrt{2\pi}.
\]
\end{example}

\begin{example}[Lerch 1894]
\[
\regprod_{n=0}^{\infty} (n+x) = \frac{\sqrt{2\pi}}{\Gamma(x)}.
\]
\end{example}

\section{Borel's result} Next, we recall Borel's result \cite{B} concerning 
\[
{\rm rank} K_{n}(A) = {\rm dim} K_{n}(A)_{\mathbb{C}}
\]
for $n=0,1,2,...$. It is given as:
\[
{\rm rank} K_{n}(A) = \left\{
\begin{array}{ccc}
1 & {\rm if} & n=0,\\
r_{1}+r_{2}-1 & {\rm if} & n=1,\\
r_{1}+r_{2} & {\rm if} & n>1 \; {\rm and} \; n \equiv 1 \; {\rm mod} \; 4,\\
r_{2} & {\rm if} & n \equiv 3 \; {\rm mod} \; 4,\\
0 & {\rm if} & {\rm otherwise}.
\end{array}
\right.
\]

Thus we know that
\[
\varphi_{A}(w,s) = (s-\frac{1}{2})^{-w} + (r_{1}+r_{2}-1) s^{-w} + (r_{1}+r_{2}) \varphi_{1}(w,s) + r_{2} \varphi_{2}(w,s),
\]
where
\[
\varphi_{1}(w,s) = \sum_{n > 1 \atop n \equiv 1 \; {\rm mod} \; 4} (\frac{n-1}{2}+s)^{-w} 
\]
and
\[
\varphi_{2}(w,s)=\sum_{n \equiv 3 \; {\rm mod} \; 4} (\frac{n-1}{2}+s)^{-w}.
\]

\section{Lemmas} We prepare calculations of two regularized products.

\begin{lemma}
\[
\regprod_{n > 1 \atop n \equiv 1 \; {\rm mod} \; 4} (\frac{n-1}{2}+s) = s^{-1} \Gamma_{\mathbb{R}}(s)^{-1} (\sqrt{2\pi})^{-s} 2 \sqrt{\pi}.
\]
\end{lemma}

\begin{proof}[Proof of Lemma 1] Since
\[
\regprod_{n > 1 \atop n \equiv 1 \; {\rm mod} \; 4} (\frac{n-1}{2}+s) = \exp(- \frac{\partial}{\partial w} \varphi_{1}(w,s)\biggl|_{w=0}),
\]
where
\begin{align*}
\varphi_{1}(w,s) &= \sum_{n > 1 \atop n \equiv 1 \; {\rm mod} \; 4} (\frac{n-1}{2}+s)^{-w}\\
&= \sum_{k=1}^{\infty} (2k+s)^{-w}\\
&= 2^{-w} \sum_{k=1}^{\infty} (k+\frac{s}{2})^{-w}\\
&= 2^{-w} \zeta(w,1+\frac{s}{2})
\end{align*}
with the standard notation of the Hurwitz zeta function
\[
\zeta(w,x)=\sum_{k=0}^{\infty} (k+x)^{-w},
\]
we get

\begin{align*}
- \frac{\partial}{\partial w} \varphi_{1}(w,s)\biggl|_{w=0} &= (\log 2)\zeta(0,1+\frac{s}{2}) - \frac{\partial}{\partial w} \zeta(w,1+\frac{s}{2})\biggl|_{w=0}\\
&=(\log 2)(\frac{1}{2}-(1+\frac{s}{2}))+\log \frac{\sqrt{2\pi}}{\Gamma(1+\frac{s}{2})}\\
&=-\frac{s+1}{2}\log 2 + \log \frac{\sqrt{2\pi}}{\Gamma(1+\frac{s}{2})}.
\end{align*}

and

\begin{align*}
\regprod_{n > 1 \atop n \equiv 1 \; {\rm mod} \; 4} (\frac{n-1}{2}+s) &= 2^{-\frac{s+1}{2}} \frac{\sqrt{2\pi}}{\Gamma(1+\frac{s}{2})}\\
&= 2^{-\frac{s}{2}} \frac{\sqrt{\pi}}{\frac{s}{2}\Gamma(\frac{s}{2})}\\
&= s^{-1} 2^{1-\frac{s}{2}} \frac{\sqrt{\pi}}{\Gamma_{\mathbb{R}}(s) \pi^{\frac{s}{2}}}\\
&= s^{-1} \Gamma_{\mathbb{R}}(s)^{-1} 2^{1-\frac{s}{2}} \pi^{\frac{1-s}{2}}\\
&= s^{-1} \Gamma_{\mathbb{R}}(s)^{-1} (\sqrt{2\pi})^{-s} 2 \sqrt{\pi}.
\end{align*}

\end{proof}

\begin{lemma}
\[
\regprod_{n \equiv 3 \; {\rm mod} \; 4} (\frac{n-1}{2}+s) = \Gamma_{\mathbb{R}}(s+1)^{-1} (\sqrt{2\pi})^{-s} \sqrt{2}.
\]
\end{lemma}

\begin{proof}[Proof of Lemma 2] We calculate
\begin{align*}
\varphi_{2}(w,s) &= \sum_{n \equiv 3 \; {\rm mod} \; 4} (\frac{n-1}{2}+s)^{-w}\\
&= \sum_{k=0}^{\infty} (2k+1+s)^{-w}\\
&= 2^{-w} \sum_{k=0}^{\infty} (k+\frac{s+1}{2})^{-w}\\
&= 2^{-w} \zeta(w,\frac{s+1}{2}).
\end{align*}

Hence, we get
\begin{align*}
\regprod_{n \equiv 3 \; {\rm mod} \; 4} (\frac{n-1}{2}+s) &= \exp ( - \frac{\partial}{\partial w} \varphi_{2}(w,s)\biggl|_{w=0} )\\
&= 2^{-\frac{s}{2}} \frac{\sqrt{2\pi}}{\Gamma(\frac{s+1}{2})}\\
&= 2^{-\frac{s}{2}} \frac{\sqrt{2\pi}}{\Gamma_{\mathbb{R}}(s+1)\pi^{\frac{s+1}{2}}}\\
&= \Gamma_{\mathbb{R}}(s+1)^{-1} (\sqrt{2\pi})^{-s} \sqrt{2}.
\end{align*}
\end{proof}

\section{Proof of Theorem 1}
\begin{proof}[Proof of Theorem 1] By using the previous calculations we obtain
\begin{align*}
&{\rm det}((s I-\mathcal{R}) | \bigoplus_{n=0}^{\infty} K_{n}(A)_{\mathbb{C}} ) \\
&= (s-\frac{1}{2}) s^{r_{1}+r_{2}-1} \biggl(\regprod_{n>1 \atop n \equiv 1 \; {\rm mod} \; 4} (\frac{n-1}{2}+s)\biggl)^{r_{1}+r_{2}} \biggl(\regprod_{n \equiv 3 \; {\rm mod} \; 4} (\frac{n-1}{2}+s)\biggl)^{r_{2}}\\
&= (s-\frac{1}{2}) s^{r_{1}+r_{2}-1} \biggl( s^{-1} \Gamma_{\mathbb{R}}(s)^{-1} (\sqrt{2\pi})^{-s} 2 \sqrt{\pi} \biggl)^{r_{1}+r_{2}} \biggl( \Gamma_{\mathbb{R}}(s+1)^{-1} (\sqrt{2\pi})^{-s} \sqrt{2} \biggl)^{r_{2}}\\
&= (s-\frac{1}{2}) s^{r_{1}+r_{2}-1} s^{-r_{1}-r_{2}} \Gamma_{\mathbb{R}}(s)^{-r_{1}-r_{2}} \Gamma_{\mathbb{R}}(s+1)^{-r_{2}} (\sqrt{2\pi})^{-(r_{1}+2r_{2})s} (2\sqrt{\pi})^{r_{1}} (2\sqrt{2\pi})^{r_{2}}\\
&= (s-\frac{1}{2}) s^{-1} \Gamma_{\mathbb{R}}(s)^{-r_{1}} \biggl(\Gamma_{\mathbb{R}}(s)\Gamma_{\mathbb{R}}(s+1)\biggl)^{-r_{2}} (\sqrt{2\pi})^{-[K:\mathbb{Q}]s} C(K)\\
&= (s-\frac{1}{2}) s^{-1} \biggl( \Gamma_{\mathbb{R}}(s)^{r_{1}} \Gamma_{\mathbb{C}}(s)^{r_{2}} (\sqrt{2\pi})^{[K:\mathbb{Q}]s} \biggl)^{-1} C(K),
\end{align*}
where we used that
\begin{align*}
\Gamma_{\mathbb{C}}(s) &= \Gamma_{\mathbb{R}}(s) \Gamma_{\mathbb{R}}(s+1),\\
r_{1}+2r_{2} &= [K:\mathbb{Q}]
\end{align*}
and
\[
C(K) = (2\sqrt{\pi})^{r_{1}} (2\sqrt{2\pi})^{r_{2}} \notin \overline{\mathbb{Q}}.
\]

\end{proof}

\end{document}